\documentclass{article}     
\usepackage{amsthm}       
\def\qed{\hfill$\Box$}
\usepackage{amssymb} 
\usepackage{amsmath} 
\begin{document}
 \newcommand*{\al}{\alpha}
 \newcommand*{\ba}{\beta}
 \newcommand*{\Da}{\Delta}
 \newcommand*{\da}{\delta}
 \newcommand*{\Ga}{\Gamma}
 \newcommand*{\ga}{\gamma}
 \newcommand*{\ka}{\kappa}
 \newcommand*{\wk}{{\wt{\ka}}}
 \newcommand*{\kwk}{{\ka\wt{\ka}}}
 \newcommand*{\lda}{\lambda}
 \newcommand*{\na}{\nabla}
 \newcommand*{\Om}{\Omega}
 \newcommand*{\oO}{\ol{\Omega}} 
 \newcommand*{\pO}{{\pl\Omega}}
 \newcommand*{\Sa}{\Sigma}
 \newcommand*{\sa}{\sigma}
 \newcommand*{\za}{\zeta}
 \newcommand*{\ve}{\varepsilon}
 \newcommand*{\vp}{\varphi}
 \newcommand*{\AB}{(\cA,\cB)}
 \newcommand*{\EeEn}{(E_1,E_0)}
 \newcommand*{\EnEe}{(E_0,E_1)}
 \newcommand*{\EnEn}{(E_0,E_0)}
 \newcommand*{\JEe}{(J,E_1)}
 \newcommand*{\JEn}{(J,E_0)}
 \newcommand*{\Mg}{(M,g)}
 \newcommand*{\Migi}{(M_i,g_i)}
 \newcommand*{\MgK}{(M,\gK)}
 \newcommand*{\MRm}{(M,\BR^m)}
 \newcommand*{\SRn}{(S,\BR^n)}
 \newcommand*{\Vg}{(V,g)}
 \newcommand*{\BUC}{BU\kern-.3ex C}
 \newcommand*{\BH}{{\mathbb H}}
 \newcommand*{\BN}{{\mathbb N}}
 \newcommand*{\BR}{{\mathbb R}}
 \newcommand*{\BZ}{{\mathbb Z}}
 \newcommand*{\cA}{{\mathcal A}}
 \newcommand*{\cB}{{\mathcal B}}
 \newcommand*{\cD}{{\mathcal D}}
 \newcommand*{\cL}{{\mathcal L}}
 \newcommand*{\gK}{{\mathfrak K}}
 \newcommand*{\bal}{\begin{aligned}}
 \newcommand*{\eal}{\end{aligned}}
 \newcommand*{\qa}{,\qquad}
 \newcommand*{\qb}{,\quad}
 \newcommand*{\ci}{\mathaccent"7017 }   
 \newcommand*{\hb}[1]{\hbox{$#1$}} 
 \newcommand*{\sdot}{\!\cdot\!}
 \newcommand*{\sn}{\kern1pt|\kern1pt}
 \newcommand*{\bsn}{\kern1pt\big|\kern1pt}
 \newcommand*{\ssm}{\!\setminus\!}
 \newcommand*{\vsdot}{\hbox{$\vert\sdot\vert$}}
 \newcommand*{\Vsdot}{\hbox{$\Vert\sdot\Vert$}}
 \newcommand*{\bssm}{\,\big\backslash\,}
 \newcommand*{\npbd}{\postdisplaypenalty=10000} 
 \newcommand*{\prsn}{\hbox{$(\cdot\sn\cdot)$}}
 \newcommand*{\pw}{\hbox{$\dl{}\sdot{},{}\sdot{}\dr$}}
 \newcommand*{\ol}{\overline}
 \newcommand*{\ul}{\underline}
 \newcommand*{\wh}{\widehat}
 \newcommand*{\wt}{\widetilde}
 \newcommand*{\ph}{\phantom}
 \newcommand*{\hr}{\hookrightarrow}
 \newcommand*{\ra}{\rightarrow}
 \newcommand*{\dl}{\langle}
 \newcommand*{\dr}{\rangle}
 \newcommand*{\sdh}{\stackrel{d}{\hookrightarrow}}
 \newcommand*{\dist}{{\rm dist}}
 \newcommand*{\id}{{\rm id}}
 \newcommand*{\vol}{{\rm vol}}
 \newcommand*{\Lis}{{\mathcal L}{\rm is}}
 \newcommand*{\loc}{{\rm loc}}
 \newcommand*{\is}{\subset}
 \newcommand*{\bt}{\bullet}
 \newcommand*{\es}{\emptyset}
 \newcommand*{\iy}{\infty}
 \newcommand*{\mt}{\mapsto}
 \newcommand*{\nag}{\na_{\cona g}}
 \newcommand*{\pl}{\partial}
 \newcommand*{\cona}{\kern-1pt}
 \newcommand*{\coU}{\kern-1pt}
 \newcommand*{\coW}{\kern-1pt}
 \makeatletter
 \newif\ifinany@
 \newcount\column@
 \def\column@plus{%
    \global\advance\column@\@ne
 }
 \newcount\maxfields@
 \def\add@amps#1{%
    \begingroup
        \count@#1
        \DN@{}%
        \loop
            \ifnum\count@>\column@
                \edef\next@{&\next@}%
                \advance\count@\m@ne
        \repeat
    \@xp\endgroup
    \next@
 }
 \def\Let@{\let\\\math@cr}
 \def\restore@math@cr{\def\math@cr@@@{\cr}}
 \restore@math@cr
 \def\default@tag{\let\tag\dft@tag}
 \default@tag

 \newbox\strutbox@
 \def\strut@{\copy\strutbox@}
 \addto@hook\every@math@size{%
  \global\setbox\strutbox@\hbox{\lower.5\normallineskiplimit
         \vbox{\kern-\normallineskiplimit\copy\strutbox}}}

\renewcommand{\start@aligned}[2]{%
    \RIfM@\else
        \nonmatherr@{\begin{\@currenvir}}%
    \fi
    \null\,%
    \if #1t\vtop \else \if#1b \vbox \else \vcenter \fi \fi \bgroup
        \maxfields@#2\relax
        \ifnum\maxfields@>\m@ne
            \multiply\maxfields@\tw@
            \let\math@cr@@@\math@cr@@@alignedat
        \else
            \restore@math@cr
        \fi
        \Let@
        \default@tag
        \ifinany@\else\openup\jot\fi
        \column@\z@
        \ialign\bgroup
           &\column@plus
            \hfil
            \strut@
            $\m@th\displaystyle{##}$%
           &\column@plus
            $\m@th\displaystyle{{}##}$%
            \hfil
            \crcr
 }
\renewenvironment{aligned}[1][c]{%
    \start@aligned{#1}\m@ne
 }{%
    \crcr\egroup\egroup
 }
 \makeatother
 \makeatletter
 \@addtoreset{equation}{section}
 \makeatother
\renewcommand{\theequation}{\arabic{section}.\arabic{equation}}        
\newtheorem{theorem}{Theorem}{\it}{\rm}
 \makeatletter
 \@addtoreset{theorem}{section}
 \makeatother
\renewcommand{\thetheorem}{\arabic{section}.\arabic{theorem}}        
 \makeatletter
 \@addtoreset{example}{section}
 \makeatother
\newtheorem{remark}{Remark}{\rm}{\rm}
 \makeatletter
 \@addtoreset{remark}{section}
 \makeatother
\renewcommand{\theremark}{{\arabic{section}.\arabic{remark}}}        
\newtheorem{remarks}{Remarks}{\rm}{\rm}
 \makeatletter
\renewcommand{\theremarks}{\arabic{section}.\arabic{remarks}}        
 \@addtoreset{remarks}{section}
 \makeatother
\newtheorem{lemma}{Lemma}{\it}{\rm}
 \makeatletter
 \@addtoreset{lemma}{section}
 \makeatother
\renewcommand{\thelemma}{\arabic{section}.\arabic{lemma}}        
\newtheorem{corollary}{Corollary}{\it}{\rm}
 \makeatletter
 \@addtoreset{corollary}{section}
 \makeatother
\renewcommand{\thecorollary}{\arabic{section}.\arabic{corollary}}        
\newtheorem{example}{Example}{\rm}{\rm}
\renewcommand{\theexample}{\arabic{section}.\arabic{example}}        
\newtheorem{examples}{Examples}{\rm}{\rm}
 \makeatletter
\renewcommand{\theexamples}{\arabic{section}.\arabic{examples}}        
 \@addtoreset{examples}{section}
 \makeatother
\newtheorem*{ExtraProof}{Proof of Theorem~1.3}{\bf}{\rm} 
\title{Linear Parabolic Equations with Strong Boundary Degeneration}
\author{Herbert Amann} 
\date{}
\maketitle
\begin{center} 
{\small 
Dedicated to Michel Chipot\\ 
in Appreciation of our Joint Professional Time} 
\end{center} 
\begin{abstract} 
As an application of the theory of linear parabolic differential equations 
on noncompact Riemannian manifolds, developed in earlier papers, we prove a 
maximal regularity  theorem for nonuniformly parabolic boundary value 
problems in Euclidean spaces. The new feature of our result is the fact 
that---besides of obtaining an optimal solution theory---we consider the 
`natural' case where the degeneration occurs only in the normal direction.  
\end{abstract}
\makeatletter
\def\blfootnote{\xdef\@thefnmark{}\@footnotetext}
\makeatother 
\setcounter{section}{0} 
\section{Introduction}\label{sec-I} 
Of concern in this paper
are linear second order parabolic differential equations which are not 
uniformly parabolic but degenerate near (some part of) the boundary. 
In the main body of this work such equations are studied in the 
framework of Riemannian manifolds. Here we restrict ourselves to a simpler 
Euclidean setting. 

\par 
We assume that $\Om$~is a bounded domain in~$\BR^m$, 
\ \hb{m\geq1}, with a smooth boundary~$\pO$ which lies locally on one side 
of~$\Om$. We write 
\begin{equation}\label{I.dO} 
\pO=\Ga\cup\Ga_0\cup\Ga_1, 
\end{equation}  
where $\Ga$, $\Ga_0$, and~$\Ga_1$ are pairwise disjoint and open and closed 
in~$\pO$ with 
\hb{\Ga\neq\es}. Either $\Ga_0$ or~$\Ga_1$, or both, may be empty in which 
case obvious adaptions apply (as is the case if 
\hb{m=1}). We denote by~$\nu$ the inner (unit) normal on~$\pO$ and 
by~$\ga$ the trace operator 
\hb{u\mt u\sn\pO}. By~%
\hb{\cdot} or~%
\hb{\prsn} we denominate the Euclidean inner product in~$\BR^m$ and 
\hb{:}~stands for the Hilbert--Schmidt inner product 
in~$\BR^{m\times m}$. Moreover, $\na u$~is the \hbox{$m$-vector} of first 
order derivatives, and $\na^2u$~is the 
\hb{(m\times m)}-matrix of second order derivatives. As usual, $C^k$~is used 
for spaces of \hbox{$C^k$~functions}, $B$~stands for `bounded', and 
$\BUC$~for `bounded and uniformly continuous'.  

\bigskip 
We set 
$$ M:=\oO\ssm\Ga 
$$ 
and consider on~$M$ a second order linear boundary value (BVP), 
denoted by $\AB$, where 
$$ 
\cA u 
:=-a:\na^2u+a_1\cdot\na u+a_0u\quad\text{on }M 
$$ 
and 
$$ 
\cB u
:= 
\left\{ 
\bal 
{}
&\ga u                      &&\quad\text{on }\Ga_0,\cr 
&b\cdot\ga\na u+b_0\ga u    &&\quad\text{on }\Ga_1
\eal 
\right. 
$$ 
on 
\hb{\pl M:=\Ga_0\cup\Ga_1}. It is assumed that 
\begin{equation}\label{I.aC} 
a=a^*\in C(M,\BR^{m\times m}) 
\qb a_1\in C\MRm 
\qb a_0\in C(M),
\end{equation}  
and 
$$
b\in BC^1(\Ga_1,\BR^m) 
\qb b_0\in BC^1(\Ga_1).  
$$ 
We also suppose that $\cA$~is strongly elliptic, that is, there exists 
\hb{\ul{\al}\colon M\ra(0,1]} such that 
$$ 
\bigl(a(x)\xi\bsn\xi\bigr)\geq\ul{\al}(x)\,|\xi|^2
\qa x\in M, 
$$ 
and that $\cB$~is normal, which means 
\begin{equation}\label{I.bn} 
\big|\bigl(b(x)\bsn\nu(x)\bigr)\big|>0 
\qa x\in\Ga_1. 
\end{equation}  
Note that $\cB$~is the Dirichlet boundary operator on~$\Ga_0$ and a 
first order boundary operator on~$\Ga_1$. 

\bigskip 
We fix 
\hb{T\in(0,\iy)} and set 
\hb{J:=[0,T]}. In this paper we develop an $L_p$~So\-bo\-lev 
space theory for the parabolic BVP on 
\hb{M\times J}:
\begin{equation}\label{I.P} 
\bal 
\pl_tu+\cA u &=f    &&\quad\text{on }    &M      &\times J, \cr
       \cB u &=0    &&\quad\text{on }    &\pl M  &\times J, \cr
      \ga_0u &=u_0  &&\quad\text{on }    &M      &\times\{0\}, 
\eal 
\end{equation}  
where $\ga_0$~is the trace operator at 
\hb{t=0}. Observe that \eqref{I.P} is \emph{not} a~BVP on~$\oO$, since there 
is no boundary condition on~$\Ga$. Also note that $\cA$~is 
\emph{not} assumed to be uniformly elliptic. 

\par 
In general, \eqref{I.P} will not be well-posed. We now introduce conditions 
for the behavior of $a$ and~$a_1$ near~$\Ga$ which guarantee an optimal 
solvability theory. This is done by prescribing---by means of a singularity 
function---the way by which $a$ and~$a_1$ vanish as we approach~$\Ga$. 

\par 
We call a function 
$$ 
R\in C^\iy\bigl((0,1],(0,\iy)\bigr) 
\quad\text{with}\quad 
\int_0^1\frac{dy}{R(y)}=\iy 
\npbd 
$$ 
(\emph{strong}) \emph{singularity function}. 
\begin{example}\label{exa-I.S1} 
{\rm 
Suppose 
\hb{s\in\BR}. Then the power function 
\hb{R_s:=(y\mt y^s)} is a strong singularity function iff 
\hb{s\geq1}. Also 
\hb{y\mt e^{-\ba y^{-\ga}}} is a strong singularity function if 
\hb{\ba,\ga>0}.}\qed 
\end{example} 

\bigskip 
To specify the behavior of the coefficients of~$\cA$ near~$\Ga$ 
we choose a normal collar for it. This means that we fix 
\hb{0<\ve\leq1} such that, setting 
$$ 
S:=\bigl\{\,q+y\nu(q)\ ;\ 0<y\leq\ve,\ q\in\Ga\,\bigr\}, 
$$ 
the map 
\begin{equation}\label{I.col} 
\vp\colon\ol{S}\ra[0,\ve]\times\Ga 
\qb q+y\nu(q)\mt(y,q) 
\end{equation}  
is a smooth diffeomorphism. Hence 
$$ 
y=\dist(x,\Ga) 
\quad\text{for}\quad 
x=q+y\nu(q)\in S. 
$$ 
We select 
\hb{\rho\in C^\iy\bigl(M,(0,1]\bigr)} satisfying 
\hb{\rho(x)=\dist(x,\Ga)} for 
\hb{x\in S} and set 
$$ 
r(x):=R\bigl(\rho(x)\bigr) 
\qa x\in M. 
$$ 
We also define 
\hb{\nu\in C^\iy\SRn} by extending the normal vector field from~$\Ga$ to~$S$ 
by setting 
$$ 
\nu(x):=\nu(q) 
\qa x=q+y\nu(q)\in S. 
$$ 

\par 
The operator~$\cA$ is said to be \hbox{$R$\emph{-degenerate}} 
\emph{uniformly strongly elliptic on}~$M$ if 
\begin{equation}\label{I.dA} 
\bal 
{\rm(i)}\quad 
&\cA\text{ is strongly elliptic on }M;\cr 
{\rm(ii)}\quad 
&\text{there exists $\ul{\al}\in(0,1)$ such that}\cr 
&\bigl(a(x)\xi\bsn\xi\bigr) 
 \geq\ul{\al}\bigl(r^2(x)\eta^2+|\za|^2\bigr)\cr  
&\text{for all $x\in S$ and $\xi=\eta\nu(x)+\za\in\BR^m$ with } 
 \bigl(\za\bsn\nu(x)\bigr)=0.    
\eal 
\end{equation}  
The boundary value problem~$\AB$ is called 
\hbox{$R$\emph{-degenerate}} \emph{uniformly strongly elliptic on}~$M$ if 
$\cA$~has this property and $\cB$~is normal. It is 
\emph{strongly degenerate near}~$\Ga$ if \eqref{I.dA} holds for some 
singularity function~$R$. 

\bigskip 
Let 
\hb{\lda\colon V_\lda\ra\BR^{m-1}}, 
\ \hb{q\mt z=(z^2,\ldots,z^m)} be a local coordinate system for~$\Ga$. Set 
\hb{U_{\coU\ka}:=\vp^{-1}\bigl([0,\ve)\times V_\lda\bigr)\is\oO}. Then 
\begin{equation}\label{I.k} 
\ka:=(\id_{[0,\ve)}\times\lda)\circ\vp 
\colon U_{\coU\ka}\ra\BH^m:=\BR_+\times\BR^{m-1}
\end{equation}  
is a local boundary flattening chart for~$\oO$. It follows from 
Section~\ref{sec-E} that~$\cA_\ka$, the local representation of 
\hb{\cA\sn U_{\coU\ka}} in the coordinate system 
\hb{\ka=(y,z)}, is given by  
\begin{equation}\label{I.Ak} 
\bal 
\cA_\ka 
&=-\bigl(\ol{a}_\ka^{11}(R\pl_y)^2+2\ol{a}_\ka^{1\al}(R\pl_y)\pl_{z^\al} 
 +\ol{a}_\ka^{\al\ba}\pl_{z^\al}\pl_{z^\ba}\bigr)\cr 
&\ph{={}} 
 +\ol{a}_\ka^1(R\pl_y)+\ol{a}_\ka^\al\pl_{z^\al}+\ol{a}_\ka^0, 
\eal 
\end{equation}  
where we use the summation convention with $\al$ and~$\ba$ running 
from~$2$ to~$m$. The operator~$\cA_\ka$ on 
\hb{\ka(U_{\coU\ka})\is\BH^m} is \hbox{$bc$\emph{-regular}} if 
\begin{equation}\label{I.aa} 
\ol{a}_\ka^{ij}\in\BUC\bigl(\ka(U_{\coU\ka})\bigr) 
\qb \ol{a}_\ka^k,\ol{a}_\ka^0\in BC\bigl(\ka(U_{\coU\ka})\bigr) 
\qa 1\leq i,j,k\leq m. 
\end{equation}  
We call~$\cA$\, \hbox{$R$\emph{-degenerate}} \hbox{$bc$-\emph{regular}} if  
$$ 
\bal 
{\rm(i)}\quad 
&\eqref{I.aC}\text{ applies};\cr 
{\rm(ii)}\quad 
&\cA_\ka\text{ is $bc$-regular for each boundary flattening chart 
 of the form \eqref{I.k}}. 
\eal 
$$ 
\addtocounter{remark}{1}
\begin{remark}\label{rem-I.bc} 
{\rm 
The ellipticity condition~\eqref{I.dA}(ii) is equivalent to the 
statement:\newline  
for each~$\ka$ of the from~\eqref{I.k}, the 
matrix~$\big[\ol{a}_\ka^{ij}\big]$ 
is symmetric and uniformly positive definite on~$\ka(U_{\coU\ka})$.}\qed 
\end{remark} 

\bigskip 
Next we introduce weighted Sobolev spaces which are adapted to strongly 
degenerate differential operators. We assume throughout 
$$ 
\bt\quad 
1<p<\iy. 
$$ 
The representation of 
\hb{u\colon S\ra\BR} in the variables~$(y,q)$ is denoted by~$\vp_*u$, 
that is, 
\hb{\vp_*u=u\circ\vp^{-1}}. Given 
\hb{k=0,1,2} and 
\hb{u\in C^2(S)}, 
$$ 
\|u\|_{W_{\coW p}^k(S;R)} 
:=\sum_{i=0}^k\Bigl(\int_0^\ve\bigl\|\bigl(R(y)\pl_y\bigr)^i 
\vp_*u(y,\cdot)\bigr\|_{W_{\coW p}^{k-i}(\Ga)}^p 
\,\frac{dy}{R(y)}\Bigr)^{1/p}. 
$$ 
The Sobolev space~$W_{\coW p}^k(S;R)$ is defined to be the completion 
in~$L_{1,\loc}(S)$ of the set of smooth compactly supported functions 
with respect to this norm. 

\par 
We choose a relatively compact open subset~$U$ of~$M$ such that 
\hb{S\cup U=M}. Then the \emph{Sobolev space}~$W_{\coW p}^k(M;R)$ consists 
of all 
\hb{u\in L_{1,\loc}(M)} for which 
$$ 
u\sn S\in W_{\coW p}^k(S;R) 
\qb u\sn U\in W_{\coW p}^k(U).
$$ 
It is a Banach space with the norm 
$$ 
u\mt\big\|u\sn S\big\|_{W_{\coW p}^k(S;R)} 
+\big\|u\sn U\big\|_{W_{\coW p}^k(U)},  
\npbd 
$$ 
whose topology is independent of the particular choice of $S$ and~$U$. 

\bigskip 
For a concise formulation of our solvability result for problem~\eqref{I.P} 
we recall some notation. Given Banach spaces $E_0$ and~$E_1$, 
\ $\cL\EeEn$~is the Banach space of bounded linear operators from~$E_1$ 
into~$E_0$, and $\Lis\EeEn$ is the set of isomorphisms therein. As usual, 
\hb{E_1\hr E_0} means that $E_1$~is continuously injected in~$E_0$, and 
\hb{E_1\sdh E_0} indicates that $E_1$~is also dense in~$E_0$. We write 
$E_{1-1/p}$ for the real interpolation space~$\EnEe_{1-1/p,p}$. 

\par 
Suppose 
\hb{E_1\sdh E_0} and 
\hb{A\in\cL\EeEn}. Then $A$~is said to have \emph{maximal 
\hbox{$L_p$~regularity}} if, for each 
\hb{(f,u_0)\in L_p\JEn\times E_{1-1/p}}, the linear evolution equation 
in~$E_0$, 
$$ 
\pl u+Au=f\text{ on }J 
\qb \ga_0u=u_0, 
$$ 
has a~unique solution 
\hb{u\in L_p\JEe\cap W_{\coW p}^1\JEn} depending continuously on $(f,u_0)$. 
By Banach's homomorphism theorem this is equivalent to 
$$ 
(\pl+A,\ \ga_0) 
\in\Lis\bigl(L_p\JEe\cap W_{\coW p}^1\JEn, 
\ L_p\JEn\times E_{1-1/p}\bigr). 
\npbd 
$$ 
This concept is independent of~$T$. 

\par 
Henceforth, we express maximal \hbox{$L_p$~regularity} more precisely 
by saying 
$$ 
\bigl(L_p\JEe\cap W_{\coW p}^1\JEn,\ L_p\JEn\bigr) 
$$ 
is~a \emph{pair of maximal regularity for}~$A$. It is known that this 
condition implies that~$-A$, considered as a linear operator in~$E_0$ with 
domain~$E_1$, generates a strongly continuous analytic semigroup on~$E_0$, 
that is, in 
\hb{\cL(E_0)=\cL\EnEn}. 
For all this we refer to Chapter~III in~\cite{Ama95a}. 

\bigskip 
We suppose:  
\begin{equation}\label{I.ass} 
\bal 
{\rm(i)}\quad 
&R\text{ is  a strong singularity function}.\cr 
{\rm(ii)}\quad 
&\AB\text{ is an $R$-degenerate}\cr 
\noalign{\vskip-1\jot} 
&\text{uniformly strongly elliptic  BVP on }M.\cr   
{\rm(iii)}\quad 
&\cA\text{ is $R$-degenerate $bc$-regular}. 
\eal 
\end{equation} 
Then 
$$ 
W_{\coW p,\cB}^2(M;R) 
:=\bigl\{\,u\in W_{\coW p}^2(M;R)\ ;\ \cB u=0\,\bigr\} 
$$  
is a closed linear subspace of~$W_{\coW p}^2(M;R)$, 
$$ 
W_{\coW p,\cB}^2(M;R)\sdh L_p(M;R) 
:=W_{\coW p}^0(M;R), 
$$ 
and 
$$ 
A:=\cA\sn W_{\coW p,\cB}^2(M;R) 
\in\cL\bigl(W_{\coW p,\cB}^2(M;R),L_p(M;R)\bigr). 
$$ 
Hence the parabolic BVP~\eqref{I.P} can be interpreted, using standard 
identifications, as the linear evolution equation in~$L_p(M;R)$: 
$$ 
\pl u+Au=f\text{ on }J 
\qb \ga_0u=u_0. 
$$ 
Now we can formulate our well-posedness result for~\eqref{I.P}. 
\addtocounter{theorem}{2} 
\begin{theorem}\label{thm-I.MR} 
Let \eqref{I.ass} be satisfied. Then 
$$ 
\bigl(L_p(J,W_{\coW p}^2(M;R))\cap W_{\coW p}^1(J,L_p(M;R)), 
\ L_p(J,L_p(M;R))\bigr) 
\npbd 
$$ 
is a pair of maximal regularity for~$A$. 
\end{theorem} 
\begin{proof} 
See Section~\ref{sec-E}. 
\end{proof} 
\addtocounter{remarks}{3} 
\begin{remarks}\label{rem-I.R}  
{\rm 
(a) 
This theorem has an obvious generalization to situations in which 
$R$~varies from connected component to connected component of~$\Ga$. It 
also applies verbatim to strongly elliptic systems. 

\par 
(b) 
The weighted Sobolev space~$W_{\coW p}^2(M;R)$ satisfies 
embedding theorems analogous to the familiar ones for the unweighted 
spaces~$W_{\coW p}^2(M)$. This implies, in particular, that the solution~$u$ 
and its first derivatives are H\"older continuous if 
\hb{p>m}. We refrain from giving details, since we would need to introduce 
appropriately weighted H\"older spaces. 

\par 
It is also possible to establish a H\"older space analog of 
Theorem~\ref{thm-I.MR}, as well as optimal solvability results for 
nonautonomous problems in parabolic space-time settings of the type 
\hb{W_{\coW p}^{2,1}(M\times J;R)}. All this will be found in the 
forthcoming book~\cite{AmaVolIII21a}. 

\par 
(c) 
For simplicity, we have restricted ourselves to bounded domains. However, 
Theorem~\ref{thm-I.MR} remains valid if it is only assumed that 
$\pO$~is uniformly regular in the sense of F.E.~Browder~\cite{Bro59a} 
(also see \cite[IV.\S4]{LSU68a} and Section~\ref{sec-U} below).}\qed 
\end{remarks} 

\bigskip 
It is worthwhile to have a closer look at a simple model problem, taking the 
last remark into account. 
\addtocounter{example}{3}
\begin{example}\label{exa-I.ex} 
{\rm 
Let 
\hb{\oO:=[0,1]\times\BR^{m-1}}. Then $\pO$~is the union of 
\hb{\pl_0\Om\cup\pl_1\Om} with 
\hb{\pl_i\Om=\{i\}\times\BR^{m-1}}. Set 
\hb{\Ga:=\pl_0\Om} (identified with~$\BR^{m-1}$) and fix 
\hb{s\geq1}. On 
\hb{M:=(0,1]\times\BR^{m-1}} consider the Dirichlet BVP $(\cA_s,\ga)$ with 
\begin{equation}\label{I.Asl} 
\bal 
\cA_s 
&:=-\bigl(y^s\pl_y(y^s\pl_y)+\Da_{m-1}\bigr)\cr 
&\phantom{:}=-(y^{2s}\pl_y^2+\Da_{m-1}) 
 -(sy^{s-1})y^s\pl_y, 
\eal
\end{equation} 
where $\Da_{m-1}$~is the Laplace operator on~$\BR^{m-1}$. Since 
\hb{|sy^{s-1}|\leq s} on~$M$, it is obvious that $(\cA_s,\ga)$ is 
\hbox{$R_s$-degenerate} strongly uniformly elliptic on~$M$. Here we can take 
\hb{S=M}. Note that 
$$ 
L_p(M;R_s)=L_p(M,y^{-s}dy\,dz) 
\npbd 
$$ 
with 
\hb{z\in\BR^{m-1}}. 

\par 
The operator~$\cA_s$ can be rewritten as 
\begin{equation}\label{I.As} 
\cA_s=-(y^{2s}\Da_{g_s}+\Da_{m-1}), 
\end{equation}  
$\Da_{g_s}$~being the Laplace--Beltrami operator on~$(0,1]$ for 
the metric 
\hb{g_s=y^{-2s}\,dy^2} (see \eqref{R.na}).}\qed 
\end{example} 

\bigskip 
The interpretation~\eqref{I.As} is the first pivotal step on the way to an 
efficient and successful handling of strongly degenerate parabolic BVPs. 
The second step, which takes the theory off the ground, is the proof 
(in Section~\ref{sec-B}) that 
\hb{\bigl((0,1],g_s\bigr)} is a uniformly regular Riemannian manifold 
(in the sense of Section~\ref{sec-U}). 

\par 
Although there has been done much work on degenerate parabolic 
differential equations, there are only very few papers known to us 
dealing with strong boundary degenerations. We mention, in particular, 
A.V.~Fursikov~\cite{Fur69a}, V.~Vespri~\cite{Ves89b}, 
N.V.~Krylov~\cite{Kry99a}, N.V.~Krylov and S.V.~Lototsky~\cite{KryL99b}, 
S.V.~Lototsky~\cite{Lot00a}, \hbox{K.-H.}~Kim~\cite{Kim07a}, and 
S.~Fornaro, G.~Metafune, and D.~Pallara~\cite{FMP11a}. 
In all but \cite{Fur69a}, \cite{Kry99a}, and \cite{KryL99b}, 
uniform boundary degenerations of type~$R_s$, 
\ \hb{s\geq1}, are being considered. This means that 
the ellipticity condition 
\begin{equation}\label{I.ar} 
\bigl(a(x)\xi\bsn\xi\bigr) 
\geq\ul{\al}\rho^{2s}(x)\,|\xi|^2 
\qa x\in M, 
\end{equation}  
is imposed. Vespri, Fornaro et~al., and also Kim, consider the operator 
\begin{equation}\label{I.Asr} 
\wt{\cA}_su:=-\rho^{2s}a\colon\na^2u+\rho^sa_1\cdot\na u+a_0u 
\end{equation}  
in 
\hb{M=\Om}, which means that 
\hb{\Ga=\pO}, with smooth coefficients, and a uniformly positive definite 
diffusion matrix~$a$. They study~$\wt{A}_s$ on the weighted Sobolev space 
$$ 
\wt{W}_{\coW p}^2(\Om;\rho^s) 
:=
\bigl\{\,u\in L_p(\Om) 
\ ;\ \rho^s\pl_iu,\rho^{2s}\pl_j\pl_ku\in L_p(\Om), 
\ 1\leq i,j,k\leq m\,\bigr\}. 
$$ 
Fornaro et~al.\ give a new, functional analytically based proof 
for Vespri's result which says that $-\wt{A}_s$~generates a strongly 
continuous analytic semigroup on~$L_p(\Om)$. In a preparatory step they 
consider, in the setting of Example~{\ref{exa-I.ex}}, the operator 
\begin{equation}\label{I.Ds} 
-y^{2s}\Da_m+y^sa_1\cdot\na 
\end{equation}  
with a constant vector~$a_1$ and show that it has maximal 
\hbox{$L_p(M)$~regularity}. That proof uses the fact that second order 
equations are considered. It is not applicable to systems or higher order 
problems. There is no maximal regularity result for the general case. 
It should be mentioned that Vespri studies H\"older space settings 
also.

\par 
Kim~\cite{Kim07a} proves a maximal regularity theorem by employing weighted 
Bessel potential spaces, introduced originally by 
N.V.~Krylov~\cite{Kry99b}, \cite{Kry99a} in connection with stochastic 
evolution equations. Krylov considers the half-space~$\BH^m$ and 
\hb{s=1}, and uses basically the fact that a logarithmic change of variables 
reduces the weighted spaces to the standard Bessel potential spaces 
on~$\BR^m$. Kim's proof is in the spirit of the classical theory of partial 
differential equations. He employs a~priori estimates due to
Krylov~\cite{Kry99a} and versions of the Krylov spaces for bounded 
domains, established by  S.V.~Lototsky~\cite{Lot00a}. A~similar approach 
is used by the latter author for a related degenerate operator. However, 
Lototsky builds on techniques from stochastic differential equations. 

\par 
Parabolic equations with strong boundary degeneration occur, in particular, 
in connection with Ito stochastic parabolic equations (e.g., 
Lototsky~\cite{Lot01a}, Krylov and Lototsky \cite{KryL99b}, 
\cite{KryL99a}, \hbox{K.-H.}~Kim and N.V.~Krylov \cite{KimK04a}, 
\cite{KimK04b}, and the references therein). 

\par 
The obvious difference between \eqref{I.ar} (resp.~\eqref{I.Asr}) and 
\eqref{I.dA} is the fact that, in the former case, the diffusion and drift 
coefficients decay uniformly in all variables, whereas in \eqref{I.dA} only 
a degeneracy in the normal direction is taken into account. 
This sticks out particularly clearly by comparing \eqref{I.Asl} with 
\eqref{I.Ds}. Our approach seems to be 
more natural since, a~priori, there is no reason to expect that tangential 
derivatives blow up near~$\Ga$. (See \cite{KryL99b} for a similar 
remark.) 

\par 
The only results for parabolic equations with degeneracies in normal 
directions only are 
in \cite{Fur69a}, \cite{Kry99a}, and~\cite{KryL99b}. Fursikov 
establishes an \hbox{$L_2$~theory} for general parabolic systems of 
arbitrary order which are of the type of Euler's differential equation. 
This means that, in the model half-space case, 
$\pl_y$~carries the weight~$y$. He uses a logarithmic change of 
variables and builds on the work of M.S.~Agranovich and 
M.I.~Vishik~\cite{AgV64a}. Krylov~\cite{Kry99a}, 
resp.\ Krylov and Lototsky~\cite{KryL99b}, establish a maximal 
regularity theory in the case of the one-dimensional half-line, 
resp.~$\BH^m$, in the weighted Bessel potential spaces introduced 
in \cite{Kry99b}, resp.~\cite{Kry99a}. 
Our paper is the first one in which the case of a general 
domain, in fact, a~general Riemannian manifold, is being handled.  

\par
Section~\ref{sec-U} contains a brief review of the relevant facts on 
uniformly regular Riemannian manifolds. In Section~\ref{sec-F} we 
present the corresponding function space settings. In the subsequent 
section we recall the maximal regularity theorem for second order uniformly 
parabolic BVPs on uniformly regular Riemannian manifolds. 

\par 
In Section~\ref{sec-B} we introduce uniformly regular Riemannian manifolds 
with strong boundary singularities. Then, in Section~\ref{sec-R}, 
we prove a renorming 
theorem for Sobolev spaces on manifolds with strong boundary singularities. 
In the final section we investigate the concepts of uniform ellipticity and 
\hbox{$bc$-regularity} in the framework of strong boundary degeneracy 
and prove Theorem~\ref{thm-I.MR}. 
\section{Uniformly Regular Riemannian Manifolds}\label{sec-U} 
In this section we recall the definition of uniformly regular Riemannian 
manifolds and collect those 
properties of which we will make use. Details can be found in 
\cite{Ama12c}, \cite{Ama12b}, \cite{Ama15a}, and in the comprehensive 
presentation~\cite{AmaVolIII21a}. Thus we shall be rather brief. 

\par  
We use standard notation from differential geometry and function space 
theory. In particular, an upper, resp.~lower, asterisk on a symbol for a 
diffeomorphism denominates the corresponding pull-back, resp.\ push-forward 
(of tensors). 

\par 
By~$c$, resp.~$c(\al)$ etc., we denote constants~%
\hb{\geq1} which can vary from occurrence  to occurrence. 

\par 
Assume $S$~is a nonempty set. On the cone of nonnegative functions on~$S$ 
we define an equivalence relation~%
\hb{{}\sim{}} by 
\hb{f\sim g} iff 
\hb{f(s)/c\leq g(s)\leq cf(s)}, 
\ \hb{s\in S}. 

\par
An \hbox{$m$-dimensional} manifold is a separable metrizable space equipped 
with an \hbox{$m$-dimensional} smooth structure. We always work in the 
smooth category. 

\par 
Let $M$ be an \hbox{$m$-dimensional} manifold with or without boundary. If 
$\ka$~is a local chart, then we use~$U_{\coU\ka}$ for its domain, the 
coordinate patch associated with~$\ka$. The chart is normalized if 
\hb{\ka(U_{\coU\ka})=Q_\ka^m}, where 
\hb{Q_\ka^m=(-1,1)^m} if 
\hb{U\is\ci M}, the interior of~$M$, and 
\hb{Q_\ka^m=[0,1)\times(-1,1)^{m-1}} otherwise. An atlas~$\gK$ is normalized 
if it consists of normalized charts. It is shrinkable if it normalized and 
there exists 
\hb{r\in(0,1)} such that 
\hb{\bigl\{\,\ka^{-1}(rQ_\ka^m)\ ;\ \ka\in\gK\,\bigr\}} is a covering 
of~$M$. It has finite multiplicity if there exists 
\hb{k\in\BN} such that any intersection of more than~$k$ coordinate 
patches is empty. 

\bigskip 
The atlas~$\gK$ is \emph{uniformly regular}~(ur) if 
\begin{equation}\label{U.K} 
\bal 
{\rm(i)}\quad 
&\text{it is shrinkable and has finite multiplicity;}\cr 
{\rm(ii)}\quad 
&\wk\circ\ka^{-1}\in\BUC^\iy\bigl(\ka(U_\kwk),\BR^m\bigr)\text{ and}\cr 
\noalign{\vskip-1\jot} 
&\|\wk\circ\ka^{-1}\|_{k,\iy}\leq c(k),  
 \ \ka,\wk\in\gK, 
 \ k\in\BN,
 \text{ where }U_\kwk:=U_{\coU\ka}\cap U_\wk.   
\eal 
\end{equation}  
Two ur atlases $\gK$ and~$\wt{\gK}$ are equivalent if 
$$ 
\bal 
{\rm(i)}\quad 
&\text{there exists $k\in\BN$ such that each coordinate patch of $\gK$}\cr 
\noalign{\vskip-1.5\jot} 
&\text{meets at most $k$ coordinate patches of $\wt{\gK}$, 
 and vice versa;}\cr 
{\rm(ii)}\quad 
&\text{condition~\eqref{U.K}(ii) holds for all } 
 (\ka,\wt{\ka})\text{ and }(\wt{\ka},\ka)\text{ belonging to } 
 \gK\times\wt{\gK}.   
\eal 
$$ 
This defines an equivalence relation in the class of all ur~atlases. An 
equivalence class thereof is a ur~structure. By~a \emph{ur~manifold} 
we mean a manifold equipped with a ur~structure. Each ur~atlas~$\gK$ 
defines a unique ur~structure, namely the equivalence class to which it 
belongs. Thus, if we need to specify the ur~structure, we write $\MgK$ 
for the ur~manifold and say its ur~structure is induced by~$\gK$. 

\par 
Let $\MgK$ be a ur~manifold. A~Riemannian metric~$g$ on~$M$ is~\emph{ur} if 
$$ 
\bal 
{\rm(i)}\quad 
&\ka_*g\sim g_m,\ \ka\in\gK;\cr 
{\rm(ii)}\quad 
&\|\ka_*g\|_{k,\iy}\leq c(k),\ \ka\in\gK,\ k\in\BN,    
\eal 
$$ 
where 
\hb{g_m:=\prsn=dx^2} is the Euclidean 
metric\footnote{As usual, we use the same symbol for a Riemannian metric 
and its restrictions to sub\-manifolds.} on~$\BR^m$ and (i)~is understood 
in the sense of quadratic forms. This concept is well-defined, 
independently of the specific~$\gK$. A~\emph{uniformly regular Riemannian} 
(urR) \emph{manifold}, 
\hb{\Mg=(M,\gK,g)}, is~a ur~manifold, 
\hb{M=\MgK}, endowed with~a urR metric. 

\par 
In the following examples we use the natural ur~structure 
(e.g.,~the product ur~structure in Example~\ref{exa-U.ex}(c)) 
if nothing is mentioned. 
\begin{examples}\label{exa-U.ex} 
{\rm 
(a) 
Each compact Riemannian manifold is~a urR manifold and its ur~structure is 
unique. 

\par 
(b) 
Let $\Om$ be a bounded domain in~$\BR^m$ with a smooth boundary such that 
$\Om$~lies locally on one side of it. Then $(\oO,g_m)$ is~a 
urR manifold. 

\par 
(c) 
If $\Migi$, 
\ \hb{i=1,2}, are urR manifolds and at most one of them has a nonempty 
boundary, then 
\hb{(M_1\times M_2,\ g_1\times g_2)} is~a urR manifold. 

\par 
(d) 
Assume $\Mg$ is~a urR manifold with a nonempty boundary. 
Denote by~$g_{\pl M}$ the Riemannian metric on~$\pl M$ induced by~$g$. Then 
$(\pl M,g_{\pl M})$ is~a urR manifold.  

\par 
(e) 
Set 
\hb{J_k:=(k-1,\ k+1)} and 
\hb{\lda_k(s):=s-k} for 
\hb{s\in J_k} and 
\hb{k\in\BZ}. Then 
\hb{\{\,\lda_k\ ;\ k\in\BZ\}} is~a ur~atlas for~$\BR$ which induces the 
\emph{canonical} ur~structure.  Its restriction 
\hb{\{\,\lda_k\sn\BR_+\ ;\ k\in\BN\,\}} is~a ur~atlas for~$\BR_+$ inducing 
the \emph{canonical} ur~structure on~$\BR_+$. Unless explicitly said 
otherwise, $\BR$~and~$\BR_+$  are always given the canonical 
ur~structure. Then $(\BR,dx^2)$ and $(\BR_+,dx^2)$ are urR manifolds. Thus 
it follows from Example~\ref{exa-U.ex}(c) that $(\BR^m,g_m)$ and 
$(\BH^m,g_m)$ are urR manifolds. 

\par 
(f) 
Let $M$ be a manifold, $N$~a topological space, and 
\hb{f\colon N\ra M} a~homeomorphism. Let $\gK$ be an atlas for~$M$. Then 
\hb{f^*\gK:=\{\,f^*\ka\ ;\ \ka\in\gK\,\}} is an atlas for~$N$ which induces 
the smooth `pull-back' structure on~$N$. If $\gK$~is ur, 
then $f^*\gK$ also is~ur. 

\par 
Suppose 
\hb{\Mg=(M,\gK,g)} is~a urR manifold. Then 
$$ 
f^*\Mg=f^*(M,\gK,g):=(N,f^*\gK,f^*g) 
\npbd 
$$ 
is~a urR manifold and the map 
\hb{f\colon(N,f^*g)\ra\Mg} is an isometric diffeomorphism.}\qed 
\end{examples} 

\bigskip 
It follows from these examples, for instance, that the cylinders 
\hb{\BR\times M_1} or 
\hb{\BR_+\times M_2}, where $M_i$~are compact Riemannian manifolds with  
\hb{\pl M_2=\es}, are 
urR manifolds. More generally, Riemannian manifolds with cylindrical 
ends are urR manifolds (see~\cite{Ama15a}, where many more examples 
are discussed). 

\par 
Without going into detail, we mention that a Riemannian manifold without 
boundary is~a urR manifold iff it has bounded geometry (see~\cite{Ama12b} 
for one half of this assertion and \cite{DSS16a} for the other half). Thus, 
for example, $(\ci\BH^m,g_m)$~is \emph{not} a urR manifold. 

\par 
A~Riemannian manifold with boundary is~a urR manifold iff it has bounded 
geometry in the sense of Th.~Schick~\cite{Schi01a} (also see 
\cite{AGN19a}, \cite{AGN19b}, \cite{AGN19c}, \cite{GSchn13a} for related 
definitions). Detailed proofs of these equivalences will be found 
in~\cite{AmaVolIII21a}. 
\section{Function Spaces}\label{sec-F} 
Let $\Mg$ be a Riemannian manifold. We consider the tensor bundles 
$$ 
T_0^1M:=TM 
\qb T_1^0M:=T^*M 
\qb T_0^0:=\BR, 
$$ 
and 
$$ 
T_\tau^\sa M:=(TM)^{\otimes\sa}\otimes(T^*M)^{\otimes\tau} 
\qa \sa,\tau\geq1,  
$$ 
endow~$T_\tau^\sa M$  with the tensor bundle metric 
\hb{g_\sa^\tau:=g^{\otimes\sa}\otimes g^{*\,\otimes\tau}}, 
\ \hb{\sa,\tau\in\BN}, and set\footnote{If $V$~is a vector bundle over~$M$, 
then $C^k(V)$~denotes the vector space of $C^k$~sections of~$V$.} 
\begin{equation}\label{F.gst} 
|a|_{g_\sa^\tau}=\sqrt{(a\sn a)_{g_\sa^\tau}} 
:=\sqrt{g_\sa^\tau(a,a)} 
\qa a\in C(T_\tau^\sa M). 
\end{equation}  
By 
\hb{\na=\nag} we denote the Levi--Civita connection and interpret it as 
covariant derivative. Then, given a smooth function~$u$ on~$M$, 
\ \hb{\na^ku\in C^\iy(T_k^0M)} is defined by 
\hb{\na^0u:=u}, 
\ \hb{\na^1u=\na u:=du}, and 
\hb{\na^{k+1}u:=\na(\na^ku)} for 
\hb{k\in\BN}. 

\par 
Let 
\hb{\ka=(x^1,\ldots,x^m)} be a local coordinate system and set 
\hb{\pl_i:=\pl/\pl x^i}. Then 
$$ 
\na^1u=\pl_iu\,dx^i 
\qb \na^2u=\na_{\cona ij}u\,dx^i\otimes dx^j 
 =(\pl_i\pl_ju-\Ga_{ij}^k\pl_ku)dx^i\otimes dx^j,
$$ 
where 
$$ 
\Ga_{ij}^k=\frac12\,g^{k\ell}(\pl_ig_{j\ell}+\pl_jg_{i\ell}-\pl_\ell g_{ij}) 
\qa 1\leq i,j,k\leq m, 
$$ 
are the Christoffel symbols. It follows that 
\begin{equation}\label{F.1} 
|\na u|_{g_0^1}^2=|\na u|_{g^*}^2=g^{ij}\pl_iu\pl_ju 
\end{equation}  
and 
\begin{equation}\label{F.2} 
|\na^2u|_{g_0^2}^2 
=g^{i_1j_1}g^{i_2j_2}\na_{\cona i_1i_2}u\na_{\cona j_1j_2}u. 
\npbd 
\end{equation}  
As usual, 
\hb{d\vol_g=\sqrt{g}\,dx} is the Riemann--Lebesgue volume element 
on~$U_{\coU\ka}$. 

\bigskip  
Let 
\hb{\sa,\tau\in\BN}, put 
\hb{V:=T_\tau^\sa M}, and write 
\hb{\vsdot_V:=\vsdot_{g_\sa^\tau}}. Then $\cD(V)$~is the linear subspace 
of~$C^\iy(V)$ of compactly supported sections. 

\par 
For 
\hb{1\leq q\leq\iy} we set 
$$ 
\|u\|_{L_q\Vg} 
:= 
\left\{ 
\bal 
{}
&\Bigl({\textstyle\int_M}|u|_V^q\,d\vol_g\Bigr)^{1/q},
&&\quad  1   &\leq q<\iy,\cr 
&\ {\textstyle\sup_M}|u|_V,
&&\quad     &q=\iy. 
\eal 
\right. 
$$ 
Then 
$$ 
L_q\Vg 
:=
\bigl(\bigl\{\,u\in L_{1,\loc}(M)\ ;\ \Vsdot_{L_q\Mg}<\iy 
\,\bigr\},\ \Vsdot_{L_q\Mg}\bigr) 
$$ 
is the usual Lebesgue space of \hbox{$L_q$~sections} of~$V$, and 
\hb{L_q\Mg=L_q\Vg} for 
\hb{V=T_0^0M=\BR}. If 
\hb{k\in\BN}, then 
$$ 
\|u\|_{W_{\coW q}^k\Vg} 
:=\sum_{j=0}^k\big\|\,|\na^jv|_{g_\sa^{\tau+j}}\big\|_{L_q\Mg} 
\qa 1\leq q<\iy, 
$$ 
and 
$$ 
\|u\|_{BC^k\Vg} 
:=\sum_{j=0}^k\big\|\,|\na^jv|_{g_\sa^{\tau+j}}\big\|_\iy. 
$$ 
Suppose 
\hb{1\leq q<\iy}. Then the Sobolev space~$W_{\coW q}^k\Vg$ is the 
completion of $\cD(V)$ in~$L_q\Vg$ with respect to the norm~%
\hb{\Vsdot_{W_{\coW q}^k\Vg}}.  

\par 
We denote by~$BC^k\Vg$ the Banach space of all sections 
\hb{u\in C^k(V)} for which 
\hb{\|u\|_{BC^k\Vg}} is finite. Then $bc^k\Vg$~is the closure 
of~$BC^{k+1}\Vg$ in the space~$BC^k\Vg$. 

\par 
In the classical Euclidean case, that is, if $\Mg$~is one of the Riemannian 
manifolds of Examples \ref{exa-U.ex}(b) or \ref{exa-U.ex}(e), it is 
well-known that the above definitions lead to the standard Sobolev spaces, 
resp.\ spaces of bounded and continuous, resp.\ bounded and uniformly 
continuous, \hbox{$F$-valued} functions, where 
\hb{F:=\BR^{m^\sa\times m^\tau}} (cf.~\cite{Ama12c} or~\cite{AmaVolIII21a}). 
\begin{theorem}\label{thm-F.C} 
Suppose $\Mg$~is a urR manifold. Then the Sobolev spaces of sections of~$V$ 
possess the same embedding, interpolation, and trace properties as their 
classical counterparts. 
\end{theorem}
\begin{proof} 
\cite{Ama12c}, \cite{Ama12b}, \cite{AmaVolIII21a} 
(also cf.~\cite{GSchn13a} for some of these results). 
\end{proof} 
It is possible and important to characterize these spaces locally. 
\begin{theorem}\label{thm-F.loc} 
Let $\Mg$ be a urR manifold, $\gK$~a ur~atlas, 
\hb{1\leq q<\iy}, and 
\hb{k\in\BN}. Then 
$$ 
\bal 
{\rm(i)}\quad 
&u\mt{\textstyle\sum_{\ka\in\gK}}\|\ka_*u\|_{W_{\coW q}^k(Q_\ka^m,F)} 
 \text{ is a norm for }W_{\coW q}^k\Vg.\cr 
{\rm(ii)}\quad 
&u\mt{\textstyle\max_{\ka\in\gK}}\|\ka_*u\|_{BC^k(Q_\ka^m,F)} 
 \text{ is a norm for }BC^k\Vg.\cr 
{\rm(iii)}\quad 
&u\in bc^k\Vg\text{ iff }\ka_*u\in\BUC^k(Q_\ka^m,F) 
 \text{ uniformly with respect to }\ka\in\gK. 
\eal 
$$ 
\end{theorem}
\begin{proof} 
\cite{AmaVolIII21a}. Also see \cite{Ama12c} and \cite{Ama12b} for similar 
assertions which, however, additionally involve partitions of unity.  
\end{proof} 
\section{Parabolic Problems on Uniformly Regular\\ 
Riemannian Manifolds}\label{sec-P} 
Let $\Mg$ be an \hbox{$m$-dimensional} Riemannian manifold. In this section 
we do not mention~$g$ in the notation for function spaces. Thus 
\hb{W_{\coW p}^k(M)=W_{\coW p}^k\Mg}, etc. 

\par 
We consider a second order differential operator~$\cA$, defined for 
\hb{u\in C^2(M)} by 
$$ 
\cA u:=-a_2\cdot\na^2u+a_1\cdot\na u+a_0\cdot u, 
$$ 
where 
$$ 
a_i\in C(T_0^iM)  
\qa i=0,1,2, 
$$ 
and 
\hb{{}\cdot{}}~denotes complete contraction. Then $\cA$~is 
\emph{uniformly strongly elliptic} if there exists 
\hb{\ul{\al}>0} such that 
\begin{equation}\label{P.el} 
a_2(p)\cdot(\xi\otimes\xi)\geq\ul{\al}\,|\xi|_{g^*(p)}^2 
\qa \xi\in T_p^*M 
\qb p\in M. 
\end{equation}  
\begin{remark}\label{rem-P.el} 
{\rm 
The following assumptions are equivalent: 
$$ 
\bal 
{\rm(i)}\quad 
&a_2\text{ is uniformly bounded and satisfies \eqref{P.el}}.\cr 
{\rm(ii)}\quad 
&a_2(p)\cdot(\xi\otimes\xi)\sim|\xi|_{g^*(p)}^2, 
 \ \xi\in T_p^*M,\ p\in M. 
\eal 
$$} 
\end{remark} 
\begin{proof} 
Let 
\hb{\bigl(H,\prsn\bigr)} be a Hilbert space and $A$~a positive semidefinite 
symmetric linear operator on~$H$. Then 
\hb{\|A\|=\sup\bigl\{\,(Ax\sn x)\ ;\ \|x\|=1\,\bigr\}}. From this the 
assertion is obvious. 
\end{proof} 
Suppose 
\hb{\pl M\neq\es}. A~first order boundary operator~$\cB_1$ is defined by 
$$ 
\cB_1:=b_1\cdot\ga\na+b_0\ga, 
$$ 
where 
$$ 
b_0\in C(\pl M) 
\qb b_1\in C(T_{\pl M}M), 
\npbd 
$$ 
with $T_{\pl M}$ being the restriction of~$TM$ to~$\pl M$. 

\bigskip 
We fix 
\hb{\da\in C\bigl(\pl M,\{0,1\}\bigr)} and set 
$$ 
\cB:=\da\cB_1+(1-\da)\ga. 
$$ 
Thus $\cB$~is the Dirichlet boundary operator on 
\hb{\pl_0M:=\da^{-1}\{0\}} and the first order boundary operator~$\cB_1$ on 
\hb{\pl_1M:=\da^{-1}(1)}. Note that 
$\pl_0M$ and~$\pl_1M$ are disjoint, open and closed in~$\pl M$, and 
\hb{\pl_0M\cup\pl_1M=\pl M}. Either $\pl_0M$ or~$\pl_1M$ may be empty. 
Also note that $\da$~is constant on the connected components of~$\pl M$. 
Then $\cB$~is~a \emph{uniformly normal} boundary operator if either 
\hb{\da=0} or 
$$ 
\inf_{q\in\pl_1M}\big|\bigl(b_1(q)\bsn\nu(q)\bigr)_{g(q)}\big|>0. 
$$ 
Finally, $\AB$ is~a \emph{uniformly normally elliptic BVP} on~$\Mg$ if 
$$ 
\bal 
{\rm(i)}\quad 
&\cA\text{ is uniformly strongly elliptic;}\cr 
{\rm(ii)}\quad 
&\cB\text{ is uniformly normal.}
\eal 
$$ 
The BVP $\AB$ is \emph{\hbox{$bc$-regular}} if 
\begin{equation}\label{P.a} 
a_2\in bc(T_0^2M) 
\qb a_1\in BC(TM) 
\qb a_0\in BC(M), 
\end{equation}  
and 
$$ 
b_1\in BC^1(T_{\pl M}M) 
\qb b_0\in BC^1(\pl M). 
$$ 

\bigskip 
Our interest in this section concerns the Sobolev space solvability of 
the BVP~\eqref{I.P} in the present setting. Assuming $\AB$ to be 
\hbox{$bc$-regular}, we set, as in the introduction,  
$$ 
W_{\coW p,\cB}^2(M) 
:=\bigl\{\,u\in W_{\coW p}^2(M)\ ;\ \cB u=0\,\bigr\} 
$$ 
and 
$$ 
A:=\cA\sn W_{\coW p,\cB}^2(M), 
\npbd 
$$ 
considered as a linear operator in~$L_p(M)$. 
\addtocounter{theorem}{1} 
\begin{theorem}\label{thm-P.MR} 
Suppose $\Mg$~is a urR manifold, 
\hb{1<p<\iy}, and $\AB$~is \hbox{$bc$-regular} 
and uniformly normally elliptic. Then 
$$ 
\bal 
{\rm(i)}\quad 
&\bigl(L_p(M),W_{\coW p,\cB}^2(M)\bigr) 
 \text{ is a densely injected Banach couple}.\cr 
{\rm(ii)}\quad 
&A\in\cL\bigl(W_{\coW p,\cB}^2(M),L_p(M)\bigr).\cr 
{\rm(iii)}\quad 
&\bigl(L_p(J,W_{\coW p,\cB}^2(M)) 
 \cap W_{\coW p}^1(J,L_p(M)),\ L_p(J,L_p(M))\bigr)\cr  
\noalign{\vskip-1\jot}  
&\text{ is a pair of maximal regularity for }A. 
\eal 
$$ 
\end{theorem} 
\begin{proof}
This is a special case of the much more general Theorem~1.2.3(i) 
of~\cite{Ama17a} (also see~\cite{AmaVolIII21a}). 
\end{proof} 
In order to reduce the technical requirements to a minimum, 
we restrict ourselves to autonomous second order problems with 
homogeneous boundary conditions. 

\par 
There are similar results applying to more general situations: $\AB$~can be 
non-autonomous, involve operators acting on vector bundles, and be of higher 
order, provided Shapiro--Lopatinskii conditions apply. Nonhomogeneous 
boundary conditions can also be admitted. Besides of the Sobolev space 
results, there is also a H\"older space solution theory of the same general 
nature. All this will be exposed in detail in~\cite{AmaVolIII21a}. 
The reader may also consult our earlier papers \cite{Ama16a} 
and~\cite{Ama17a}.  
\section{Uniformly Regular Manifolds with Boundary 
Singularities}\label{sec-B} 
Let $R$ be a strong singularity function, 
\hb{I:=(0,\ve]}, and set 
$$ 
\sa(y):=\int_y^\ve\frac{d\tau}{R(\tau)} 
\qa y\in I. 
\npbd 
$$ 
We denote the general point of~$\BR_+$ by~$s$. 
\begin{lemma}\label{lem-B.R} 
$\sa$~is a diffeomorphism from~$I$ onto~$\BR_+$ and 
\hb{\sa^*(ds^2)=dy^2/R^2}. 
\end{lemma} 
\begin{proof} 
The first assertion follows since 
\hb{\dot\sa(y)=-1/R(y)<0} for 
\hb{y\in I}. Hence 
\hb{\sa^*ds=d\sa=\dot\sa dy=-dy/R}. This implies the second claim. 
\end{proof} 
\addtocounter{corollary}{1} 
\begin{corollary}\label{cor-B.R} 
${}$
\hb{(I,\,dy^2/R^2)=\sa^*(\BR_+,ds^2)} is a urR manifold. 
\end{corollary} 
\begin{proof} 
Examples \ref{exa-U.ex}(e) and \ref{exa-U.ex}(f). 
\end{proof} 
Now we assume 
\begin{equation}\label{B.M} 
\bal 
\bt\quad 
&\Mg\text{ is a urR manifold}.\cr 
\bt\quad 
&\Ga\text{ is a nonempty open and closed subset of }\pl M.  
\eal 
\end{equation}  
By Example~\ref{exa-U.ex}(d), 
\ $(\Ga,g_\Ga)$~is a urR manifold. Thus, see \cite{Schi01a} 
or~\cite{AmaVolIII21a}, there exist 
\hb{\ve\in(0,1]} and a closed geodesic normal collar 
$$ 
\vp\colon\ol{S}\ra[0,\ve]\times\Ga. 
$$ 
This means that $\ol{S}$~is a closed neighborhood of~$\Ga$ in~$M$ 
and $\vp$~is a diffeomorphism with 
$$ 
\vp^{-1}(y,q)=\exp_q\bigl(y\nu(q)\bigr) 
\qa (y,q)\in[0,\ve]\times\Ga. 
$$ 
Hence 
\begin{equation}\label{B.v} 
v_q:=\bigl(t\mt\vp^{-1}(t,q)\bigr) 
\qa 0\leq t\leq\ve, 
\end{equation}  
is the unique geodesic starting at 
\hb{q\in\Ga} in the direction of the inward normal vector~$\nu(q)$. 
Moreover, 
$$ 
\vp_*g=g_N:=dy^2\oplus g_\Ga 
\npbd 
$$ 
is a product metric on 
\hb{T\bigl([0,\ve]\times\Ga\bigr)=T[0,\ve]\oplus T\Ga}. 
 
\par 
For 
\hb{0<r\leq1} we set 
$$ 
I(r):=(0,r\ve] 
\qb N(r):=I(r)\times\Ga 
\qb N:=N(1), 
$$ 
and 
$$ 
S(r):=\vp^{-1}\bigl(N(r)\bigr) 
\qb S:=S(1)=\ol{S}\ssm\Ga. 
$$ 
We equip~$S$ with a new metric,~$g_R$, as follows: we choose 
\hb{\chi\in C^\iy\bigl(I,[0,1]\bigr)} with 
\hb{\chi(y)=1} for 
\hb{y\leq\ve/3} and 
\hb{\chi(y)=0} for 
\hb{y\geq2\ve/3}. Then we put 
\begin{equation}\label{B.del} 
1/\da^2:=1-\chi+\chi/R^2 
\qb \ga_R:=dy^2/\da^2, 
\end{equation}  
and 
\begin{equation}\label{B.gR} 
g_R:=\vp^*(\ga_R\oplus g_\Ga). 
\end{equation}  
\addtocounter{lemma}{1} 
\begin{lemma}\label{lem-B.g} 
$(S,g_R)$~is a urR manifold and 
\hb{g_R(p)=g(p)} for 
\hb{p\in S\bssm S(2/3)}. 
\end{lemma}
\begin{proof} 
Corollary~\ref{cor-B.R} and Examples \ref{exa-U.ex}(c) 
and \ref{exa-U.ex}(d) imply that 
\hb{(N,\,\ga_R\oplus g_\Ga)} is~a urR manifold. Now the first claim follows 
by applying Example~\ref{exa-U.ex}(f). The second one is obvious. 
\end{proof} 
\addtocounter{theorem}{3} 
\begin{theorem}\label{thm-B.M} 
Let \eqref{B.M} be satisfied. Put 
\hb{\wh{M}:=M\ssm\Ga}. Define 
\begin{equation}\label{B.g} 
\wh{g}_R 
:=  
\left\{ 
\bal 
{}
&g   &&\quad \text{on }\wh{M}\ssm S,\cr 
&g_R &&\quad \text{on }S. 
\eal 
\right. 
\npbd 
\end{equation}  
Then $(\wh{M},\wh{g}_R)$ is a urR manifold. 
\end{theorem} 
\begin{proof} 
This is clear by the preceding lemma. 
\end{proof} 
For easy reference we say that $(\wh{M},\wh{g}_R)$ is an 
\hbox{$R$\emph{-singular}} \emph{model for}~$\Mg$ (\emph{near}~$\Ga$). 
Moreover, $\Mg$~is \hbox{$R$\emph{-singular}} (\emph{near}~$\Ga$), if it is 
equipped with an \hbox{$R$-singular} model. 
\section{A Renorming Theorem}\label{sec-R} 
In this section we derive a semi-local representation for the Sobolev norms 
on~$(\wh{M},\wh{g}_R)$. 

\par 
First we observe that, in the local coordinate system~$\id_{\ci I}$ 
for~$\ci I$, the Christoffel symbol of~$\na_{\cona\ga_R}$ 
equals 
\hb{-\dot\da/\da}. Hence 
\begin{equation}\label{R.na} 
\na_{\cona\ga_R}^2 
=\Bigl(\frac\pl{\pl y}\Bigr)^2+\frac{\dot\da}\da\,\frac\pl{\pl y} 
=\frac1{\da^2}\Bigl(\da\,\frac\pl{\pl y}\Bigr)^2. 
\end{equation}  
To simplify the writing, we set 
\begin{equation}\label{R.gh} 
h:=g_\Ga 
\qb \wt{g}:=\ga_R\oplus h=dy^2/\da^2\oplus h. 
\end{equation}  
It follows from \eqref{F.1} and 
\hb{\na_{\wt{\ga}}=\na_{\cona\ga_R}\oplus\na_h} that 
\begin{equation}\label{R.N1} 
|\na_{\wt{\ga}}v|_{\wt{g}^*}^2 
=\Big|\da\,\frac{\pl v}{\pl y}\Big|^2+|\na_hv|_{h^*}^2.  
\end{equation}  
Similarly, using \eqref{R.na} and 
\hb{\na_{\wt{g}}^2=\na_{\cona\ga_R}^2\oplus\na_h^2}, 
\begin{equation}\label{R.N2} 
|\na_{\wt{\ga}}^2v|_{\wt{g}_0^2}^2 
=\Big|\Bigl(\da\,\frac\pl{\pl y}\Bigr)^2v\Big|^2+|\na_h^2v|_{h_0^2}^2.  
\end{equation}  
Also note that 
\begin{equation}\label{R.sg} 
\sqrt{\wt{g}}=\sqrt{h}\big/\da. 
\end{equation}  

\par 
Each urR manifold possesses~a ur~atlas whose coordinate patches are smaller 
than any prescribed positive number (cf.~\cite[Section~3]{Ama15a} 
or~\cite{AmaVolIII21a}). Thus we can choose a ur~atlas~$\gK$ for~$M$ 
such that 
\hb{U_{\coU\ka}\is S\ssm S(1/3)} for each 
\hb{\ka\in\gK} for which $U_{\coU\ka}$~meets the boundary of 
\hb{S(2/3)}. Then we set 
$$ 
\gK(W):=\bigl\{\,\ka\in\gK 
\ ;\ U_{\coU\ka}\cap\bigl(M\ssm\ci S(2/3)\bigr)\neq\es\,\bigr\} 
$$ 
and 
$$ 
W:=\bigcup_{\ka\in\gK(W)}U_{\coU\ka}. 
$$ 

\par 
For 
\hb{k\in\BN} and 
\hb{u\in C^k(\wh{M})} we define 
$$ 
\|u\|_{k,p}(W) 
:=\sum_{\ka\in\gK(W)}\|\ka_*u\|_{W_{\coW p}^k(Q_\ka^m)} 
$$ 
and 
$$ 
\bal 
{} 
&\|u\|_{k,p}(S,R)\cr 
&\  
:=\sum_{j=0}^k\Bigl(\int_0^1 
\Bigl(\Big|\Bigl(R(y)\,\frac\pl{\pl y}\Bigr)^j\vp_*u(y,\cdot)\Big|^p 
+|\na_h^j\vp_*u(y,\cdot)|_{h_0^j}^p\Bigr) 
\,d\vol_h\frac{dy}{R(y)}\Bigr)^{1/p}. 
\eal 
$$ 
\begin{theorem}\label{thm-R.N} 
${}$
\hb{u\mt\|u\|_{k,p}(S,R)+\|u\|_{k,p}(W)} is a norm 
for~$W_{\coW p}^k(\wh{M},\wh{g}_R)$. 
\end{theorem} 
\begin{proof} 
This is a consequence of Theorem~\ref{thm-F.loc}, Lemma~\ref{lem-B.g}, 
Theorem~\ref{thm-B.M},\, \eqref{R.N1}, \eqref{R.N2}, and \eqref{R.sg}.  We 
leave it to the reader to fill in the details. 
\end{proof} 
Since, according to Section~\ref{sec-F}, the norm of 
$W_{\coW p}^k(\wh{M},\wh{g}_R)$ is defined in a coordinate free manner, 
it follows from this theorem that the topology of 
$W_{\coW p}^k(\wh{M},\wh{g}_R)$ is independent of the particular choice 
of the collar neighborhood (that is, of~$\ve$) and the cut-off 
function~$\chi$. 
\section{Elliptic Operators on Singular Manifolds}\label{sec-E} 
Let $(\wh{M},\wh{g}_R)$ be an \hbox{$R$-singular} model 
for~$\Mg$ near~$\Ga$ and set 
$$ 
\wh{g}:=\wh{g}_R 
\qb \wh{\na}:=\wh{\na}_{\wh{g}}. 
$$ 
Assume that 
$$ \wh{\cA}=\cA(\wh{\na}):=-a_2\cdot\wh{\na}^2+a_1\cdot\wh{\na}+a_0 
$$ 
is a linear differential operator on~$(\wh{M},\wh{g})$ with continuous 
coefficients. Due to Theorem~\ref{thm-B.M}, we can apply  
Theorem~\ref{thm-P.MR}, provided $\wh{\cA}$~is uniformly strongly elliptic
and \hbox{$bc$-regular} on~$(\wh{M},\wh{g})$ and $\cB$~is uniformly normal on 
\hb{\pl\wh{M}=\pl M\ssm\Ga}. It follows from the definition of~$\wh{g}$  
that~$\wh{\cA}$, considered as a differential operator on~$(\wh{M},\wh{g})$, 
has singular coefficients. It is the purpose of the following considerations 
to describe the assumptions on~$\wh{\cA}$ in this singular setting. 

\bigskip 
Recalling~\eqref{B.v}, we extend the normal vector field over~$S$ 
by setting  
$$ 
\nu(p):=\dot v_q(y)\in T_pS 
\quad\text{if }\quad 
\vp(p)=(y,q). 
$$ 
Now we define 
\hb{\nu^*(p)\in T_p^*S} by 
$$ 
\bigl\dl\nu^*(p),X\bigr\dr_p:=\bigl(\nu(p)\bsn X\bigr)_{g(p)} 
\qa X\in T_pS, 
$$ 
where 
\hb{\pw_p\colon T_p^*M\times T_pM\ra\BR} is the canonical duality pairing. 
Thus~$\nu(p)$, resp.~$\nu^*(p)$, is at 
\hb{p\in\vp^{-1}(y,q)} obtained from the normal vector~$\nu(q)$, 
resp.\ conormal vector~$\nu^*(q)$, by parallel transport along the 
geodesic curve~$v_q(t)$, 
\ \hb{0\leq t\leq y}. Hence 
$$ 
|\nu(p)|_{g(p)}=|\nu^*(p)|_{g^*(p)}=1 
\qa p\in S. 
\npbd 
$$ 
In abuse of language we call~$\nu^*$ conormal vector (field) \emph{on}~$S$. 

\bigskip 
We denote by 
\hb{\rho(p):=\dist_g(p,\Ga)} the distance in~$(S,g)$ from~$p$ to~$\Ga$. Thus 
\hb{\rho(p)=y} if 
\hb{\vp(p)=(y,q)}. Then 
$$ 
r(p):=R\bigl(\rho(p)\bigr) 
\qa p\in S. 
$$ 
For shorter writing we also set 
$$ 
w[\xi]^2:=w\cdot(\xi\otimes\xi) 
\qa w\in C(T_0^2M) 
\qb \xi\in C(T^*M). 
$$ 
\begin{theorem}\label{thm-E.a} 
$\wh{\cA}$~is uniformly strongly elliptic on~$(\wh{M},\wh{g})$ 
iff 
$$ 
a_2(p)[\xi]^2\sim|\xi|_{g^*(p)}^2 
\qa p\in\wh{M}\ssm S 
\qb \xi\in T_p^*M, 
$$ 
and 
\begin{equation}\label{E.ax} 
a_2(p)[\xi]^2\sim\bigl(r^2(p)\eta^2+|\za|_{g^*(p)}^2\bigr) 
\qa p\in S, 
\npbd 
\end{equation}  
for 
\hb{\xi=\eta\nu^*(p)+\za\in T_p^*M} with 
\hb{\za\perp\nu^*(p)}. 
\end{theorem} 
\begin{proof} 
Set 
\hb{\Ga_y:=\vp^{-1}(y\oplus\Ga)} for 
\hb{0<y\leq\ve}. Then 
\begin{equation}\label{E.T} 
T_p^*\Ga_{r(p)}=\nu^*(p)^\perp.
\end{equation}  
It follows from \eqref{B.gR} and \eqref{B.g} that it does not matter 
if we take the orthogonal complement with respect to~$g^*(p)$ or 
to~$\wh{g}^*(p)$. Thus, given 
$$ 
\xi=\eta\nu^*(p)+\za\in T_p^*S 
\quad\text{with}\quad  
\za\in\nu^*(p)^\perp, 
$$ 
we find 
$$ 
\vp_*\xi=\eta\oplus\wt{\za}\in T_{(y,q)}^*N=\BR\oplus T_q^*\Ga, 
$$ 
where 
\hb{\vp(p)=(y,q)}. We deduce from \eqref{B.del} and \eqref{R.gh} that 
\begin{equation}\label{E.gd} 
\wt{g}^*=\da^2dy^2\oplus h^*. 
\end{equation}  
Hence 
$$ 
|\vp_*\xi|_{\wt{g}^*(y,q)}^2=\da^2(y)\eta^2+|\wt{\za}|_{h^*(q)}^2. 
$$ 
Note that 
\hb{\da(y)\sim R(y)} for 
\hb{1/3\leq y\leq1}. Thus, since 
\hb{\da(y)=R(y)} if 
\hb{0<y\leq1/3}, we get 
\begin{equation}\label{E.Re} 
|\vp_*\xi|_{\wt{g}^*(y,q)}^2 
\sim\bigl(R^2(y)\eta^2+|\wt{\za}|_{h^*(q)}^2\bigr), 
\end{equation}  
uniformly with respect to~$\xi$. Observe that 
$$ 
|\xi|_{\wh{g}^*(p)}^2=\vp^*\bigl(\vp_*(|\xi|_{\wh{g}^*(p)}^2)\bigr) 
=\vp^*(|\vp_*\xi|_{\wt{g}^*(y,q)}^2). 
$$ 
From this and \eqref{E.Re} we obtain 
$$ 
|\xi|_{\wh{g}^*(p)}^2\sim\bigl(\rho^2(p)\eta^2+|\za|_{g^*(p)}^2\bigr) 
\qa \xi\in T^*S.  
\npbd 
$$ 
Now the assertion is an obvious consequence of \eqref{B.g} and 
Remark~\ref{rem-P.el}. 
\end{proof} 
We introduce tensor fields 
\hb{\wt{a}_i\in C(T_0^iN)}, 
\ \hb{i=0,1,2}, by setting 
$$ 
\bal 
\wt{a}_2(y,q)\cdot(\xi_1\otimes\xi_2) 
&:=(\vp_*a_2)(y,q)\cdot\Bigl(\Bigl(\frac{\eta_1}{R(y)}\oplus\za_1\Bigr)
 \otimes\Bigl(\frac{\eta_2}{R(y)}\oplus\za_2\Bigr)\Bigr)\cr 
&\phantom{:=\quad}
 \text{for }\xi_i=\eta_i\oplus\za_i 
 \in\BR\oplus T_q^*\Ga,\ i=1,2,\ (y,q)\in N,\cr 
\wt{a}_1(y,q)\cdot\xi 
&:=(\vp_*a_1)(y,q)\cdot\Bigl(\frac\eta{R(y)}\oplus\za\Bigr)\cr  
&\phantom{:=\quad} 
 \text{for }\xi=\eta\oplus\za\in\BR\oplus T_q^*\Ga,\ (y,q)\in N, 
\eal 
\npbd 
$$ 
and 
\hb{\wt{a}_0:=\vp_*a_0}. 
\begin{theorem}\label{thm-E.bc} 
We set 
\hb{S^c:=\wh{M}\ssm S}. Then $\wh{\cA}$~is \hbox{$bc$-regular} 
on~$(\wh{M},\wh{g})$ iff 
$$ 
\bal 
{\rm(i)}\quad 
&&a_2       &\in bc(T_0^2S^c,g),    
    &&\quad a_i\in BC(T_0^iS^c,g),\ i=0,1;\cr 
{\rm(ii)}\quad 
&&\wt{a}_2  &\in bc(T_0^2N,g_N),    
    &&\quad \wt{a}_i\in BC(T_0^iN,g_N),\ i=0,1. 
\eal 
$$ 
\end{theorem} 
\begin{proof} 
(1) Since, by~\eqref{B.g}, 
\ \hb{(T_0^kS^c,\wh{g})=(T_0^kS^c,g)} for 
\hb{k\in\BN}, we can restrict our considerations to~$S$. 

\par 
(2) 
We denote by~$\vp_*\wh{\cA}$ the push-forward of~$\wh{\cA}$ by~$\vp$. Thus 
$\vp_*\wh{\cA}$~is a linear operator on~$N$, defined by 
$$ 
(\vp_*\wh{\cA})v:=\vp_*\bigl(\wh{\cA}(\vp^*v)\bigr) 
\qa v\in C^2(N). 
$$ 
It follows that (see \eqref{B.gR}, \eqref{B.g}, and \eqref{R.gh}) 
$$ 
\vp_*\wh{\na}=\na_{\cona\vp_*\wh{g}} 
=\na_{\wt{g}}=\na_{\cona g_R}\oplus\na_h.
$$ 
Hence 
$$ 
\vp_*\wh{\cA} 
=-(\vp_*a_2)\cdot\na_{\wt{g}}^2+(\vp_*a_1)\cdot\na_{\wt{g}}+\vp_*a_0. 
$$ 
Using~\eqref{R.na}, we find 
$$ 
\bal 
\vp_*\wh{\cA} 
&=-(\vp_*a_2)\cdot  
\Bigl(\frac1{\da^2}\Bigl(\da\,\frac\pl{\pl y}\Bigr)^2\oplus\na_h^2\Bigr)\cr 
&\phantom{={}} 
+(\vp_*a_1)\cdot 
\Bigl(\frac1\da\Bigl(\da\,\frac\pl{\pl y}\Bigr)\oplus\na_h\Bigr) 
+\vp_*a_0. 
\eal 
$$ 
Note that, by~\eqref{B.del}, 
$$ 
R^2/\da^2=\chi+R^2(1-\chi) 
$$ 
and 
\hb{1/c\leq\pl^jR(y)\leq c} for 
\hb{\ve/3\leq y\leq\ve} and 
\hb{j=0,1}. Thus we can rewrite~$\vp_*\wh{\cA}$ as 
\begin{equation}\label{E.fA} 
\vp_*\wh{\cA} 
=-\wh{a}_2\cdot 
\Bigl(R\,\frac\pl{\pl y}\oplus\na_h\Bigr)^2  
+\wh{a}_1\cdot 
\Bigl(R\,\frac\pl{\pl y}\oplus\na_h\Bigr) 
+\wh{a}_0, 
\end{equation}  
where 
$$ 
\wh{a}_2\in bc(T_0^2N,\wt{g}) 
\qa \wh{a}_i\in BC(T_0^iN,\wt{g}) 
\qb i=0,1, 
$$ 
iff 
$$ 
\wt{a}_2\in bc(T_0^2N,g_N) 
\qa \wt{a}_i\in BC(T_0^iN,g_N) 
\qb i=0,1. 
$$ 
It is a consequence of the definition of~$\wt{a}_i$ that 
$$ 
\|\wt{a}_i\|_{BC(T_0^iN,g_N)} 
=\|\vp_*a_i\|_{BC(T_0^iN,\wt{g})} 
\qa i=0,1,2. 
$$ 
Consequently, we derive from \eqref{E.fA} that $\wh{\cA}$~is 
\hbox{$bc$-regular} 
on~$(S,\wh{g})$ iff assumption~(ii) is satisfied. From this and step~(1) 
we get the assertion. 
\end{proof} 
Finally, we prove Theorem~\ref{thm-I.MR} by specializing our general results 
to the specific setting of the introduction. 
\begin{ExtraProof} 
{\rm 
Example~\ref{exa-U.ex}(b) guarantees that 
\hb{\Mg:=(\oO,g_m)} is a urR manifold. It follows from Theorem~\ref{thm-R.N} 
that 
$$ 
W_{\coW p}^k(\oO\ssm\Ga;R)=W_{\coW p}^k(\wh{M},\wh{g}). 
$$ 
Theorem~\ref{thm-E.a} shows that the \hbox{$R$-degenerate} uniform strong 
ellipticity~\eqref{I.dA} implies that $\cA$~is uniformly strongly elliptic 
on~$(\wh{M},\wh{g})$. By taking the compactness of~$\Ga$ into account, we 
deduce from \eqref{I.Ak}, \eqref{I.aa}, and Theorem~\ref{thm-E.bc} that 
$\cA$~is \hbox{$bc$-regular} on~$(\wh{M},\wh{g})$. Due to~\eqref{I.bn} 
and the compactness of~$\Ga_1$, we see that $\cB$~is uniformly normal 
on~$\pl\wh{M}$. Now the assertion is implied by Theorem~\ref{thm-P.MR}.}\qed 
\end{ExtraProof} 
Remark~\ref{rem-I.bc} is an easy consequence of the proof of 
Theorem~\ref{thm-E.bc}, using once more the compactness of~$\Ga$. 

{\small

\noindent 
Herbert Amann, 
Math.\ Institut, Universit\"at Z\"urich, Winterthurerstr.~190,\\   
CH 8057 Z\"urich, Switzerland, herbert.amann@math.uzh.ch}


\begin{thebibliography}{10}
\bibitem{AgV64a}
M.S. Agranovich, M.I. Vishik.
\newblock Elliptic problems with a parameter and parabolic problems of general
  type.
\newblock {\em Russ.\ Math.\ Surveys}, {\bf 19} (1964), 53--157.

\bibitem{Ama95a}
H.~Amann.
\newblock {\em Linear and quasilinear parabolic problems. {V}ol. {I} Abstract
  linear theory}.
\newblock Birkh\"{a}user, Basel, 1995. 

\bibitem{Ama12c}
H.~Amann.
\newblock Anisotropic function spaces on singular manifolds (2012).
\newblock arXiv:1204.0606.

\bibitem{Ama12b}
H.~Amann.
\newblock Function spaces on singular manifolds.
\newblock {\em Math. Nachr.}, {\bf 286} (2012), 436--475.

\bibitem{Ama15a}
H.~Amann.
\newblock Uniformly regular and singular {R}iemannian manifolds.
\newblock In {\em Elliptic and parabolic equations}, volume 119 of {\em
  Springer Proc. Math. Stat.}, pages 1--43. Springer, Cham, 2015.

\bibitem{Ama16a}
H.~Amann.
\newblock Parabolic equations on uniformly regular {R}iemannian manifolds and
  degenerate initial boundary value problems.
\newblock In {\em Recent Developments of Mathematical Fluid Mechanics}, pages
  43--77. Birk\-h\"au\-ser, Basel, 2016.

\bibitem{Ama17a}
H.~Amann.
\newblock Cauchy problems for parabolic equations in {S}obolev-{S}lobodeckii
  and {H}\"{o}lder spaces on uniformly regular {R}iemannian manifolds.
\newblock {\em J. Evol. Equ.}, {\bf 17}(1) (2017), 51--100.

\bibitem{Ama19a}
H.~Amann.
\newblock {\em Linear and quasilinear parabolic problems. {V}ol. {II} Function
  spaces}.
\newblock Birkh\"{a}user, Basel, 2019.

\bibitem{AmaVolIII21a}
H.~Amann.
\newblock {\em Linear and quasilinear parabolic problems. {V}ol. {III}
  Differential equations}.
\newblock Birk\-h\"au\-ser, Ba\-sel, 2021.
\newblock In preparation.

\bibitem{AGN19a}
B.~Ammann, N.~Gro{\ss}e, V.~Nistor.
\newblock Analysis and boundary value problems on singular domains: an 
approach via bounded geometry.
\newblock {\em C. R. Math. Acad. Sci. Paris}, {\bf 357}(6) (2019), 487--493.

\bibitem{AGN19b}
B.~Ammann, N.~Gro{\ss}e, V.~Nistor.
\newblock The strong {L}egendre condition and the well-posedness of mixed
  {R}obin problems on manifolds with bounded geometry.
\newblock {\em Rev. Roumaine Math. Pures Appl.}, {\bf 64}(2-3) (2019), 
85--111.

\bibitem{AGN19c}
B.~Ammann, N.~Gro{\ss}e, V.~Nistor.
\newblock Well-posedness of the {L}aplacian on manifolds with boundary and
  bounded geometry.
\newblock {\em Math. Nachr.}, {\bf 292}(6) (2019), 1213--1237.

\bibitem{Bro59a}
F.E. Browder.
\newblock Estimates and existence theorems for elliptic boundary value
  problems.
\newblock {\em Proc. Nat. Acad. Sci. U.S.A.}, {\bf 45} (1959), 365--372.

\bibitem{DSS16a}
M.~Disconzi, Y.~Shao, G.~Simonett.
\newblock Remarks on uniformly regular {R}iemannian manifolds.
\newblock {\em Math. Nachr.}, {\bf 289} (2016), 232--242.

\bibitem{FMP11a}
S.~Fornaro, G.~Metafune, D.~Pallara.
\newblock Analytic semigroups generated in {$L\sp p$} by elliptic operators
  with high order degeneracy at the boundary.
\newblock {\em Note Mat.}, {\bf 31}(1) (2011), 103--116.

\bibitem{Fur69a}
A.V. Fursikov.
\newblock A certain class of degenerate elliptic operators.
\newblock {\em Mat. Sb. (N.S.)}, {\bf 79 (121)} (1969), 381--404.

\bibitem{GSchn13a}
N.~Gro{\ss}e, C.~Schneider.
\newblock Sobolev spaces on {R}iemannian manifolds with bounded geometry:
  general coordinates and traces.
\newblock {\em Math. Nachr.}, {\bf 286}(16) (2013), 1586--1613.

\bibitem{Kim07a}
K.-H. Kim.
\newblock Sobolev space theory of parabolic equations degenerating on the
  boundary of {$C^1$} domains.
\newblock {\em Comm. Partial Differential Equations}, {\bf 32}(7-9) (2007),
  1261--1280.

\bibitem{KimK04a}
K.-H. Kim, N.V. Krylov.
\newblock On {SPDE}s with variable coefficients in one space dimension.
\newblock {\em Potential Anal.}, {\bf 21}(3) (2004), 209--239.

\bibitem{KimK04b}
K.-H. Kim, N.V. Krylov.
\newblock On the {S}obolev space theory of parabolic and elliptic equations in
  {$C^1$} domains.
\newblock {\em SIAM J. Math. Anal.}, {\bf 36}(2) (2004), 618--642.

\bibitem{Kry99b}
N.V. Krylov.
\newblock Some properties of weighted {S}obolev spaces in {${\bf R}^d_+$}.
\newblock {\em Ann. Scuola Norm. Sup. Pisa Cl. Sci. (4)}, {\bf 28}(4) (1999),
  675--693.

\bibitem{Kry99a}
N.V. Krylov.
\newblock Weighted {S}obolev spaces and {L}aplace's equation and the heat
  equations in a half space.
\newblock {\em Comm. Partial Differential Equations}, {\bf 24}(9-10) (1999),
  1611--1653.

\bibitem{KryL99b}
N.V. Krylov, S.V. Lototsky.
\newblock A {S}obolev space theory of {SPDE}s with constant coefficients in a
  half space.
\newblock {\em SIAM J. Math. Anal.}, {\bf 31}(1) (1999), 19--33.

\bibitem{KryL99a}
N.V. Krylov, S.V. Lototsky.
\newblock A {S}obolev space theory of {SPDE}s with constant coefficients on a
  half line.
\newblock {\em SIAM J. Math. Anal.}, {\bf 30}(2) (1999), 298--325.

\bibitem{LSU68a}
O.A. Ladyzhenskaya, V.A. Solonnikov, N.N. Ural'ceva.
\newblock {\em Linear and Quasilinear Equations of Parabolic Type}.
\newblock Amer.\ Math.\ Soc., Transl.\ Math.\ Monographs, Providence, R.I.,
  1968.

\bibitem{Lot00a}
S.V. Lototsky.
\newblock Sobolev spaces with weights in domains and boundary value problems
  for degenerate elliptic equations.
\newblock {\em Methods Appl. Anal.}, {\bf 7}(1) (2000), 195--204.

\bibitem{Lot01a}
S.V. Lototsky.
\newblock Linear stochastic parabolic equations, degenerating on the boundary
  of a domain.
\newblock {\em Electron. J. Probab.}, {\bf 6} (2001), no. 24, 14.

\bibitem{Schi01a}
Th. Schick.
\newblock Manifolds with boundary and of bounded geometry.
\newblock {\em Math. Nachr.}, {\bf 223} (2001), 103--120.

\bibitem{Ves89b}
V.~Vespri.
\newblock Analytic semigroups, degenerate elliptic operators and applications
  to nonlinear {C}auchy problems.
\newblock {\em Ann. Mat. Pura Appl. (4)}, {\bf 155} (1989), 353--388.
\end{thebibliography}
\end{document}